\documentclass[11pt, english]{article}

\usepackage[margin= 2 cm,bottom=21mm, top= 20mm]{geometry}

\usepackage[dvipsnames]{xcolor}
\usepackage[square,sort,comma,numbers]{natbib}
\usepackage{amsthm}
\usepackage{amsmath}
\usepackage{amssymb}
\usepackage{setspace}
\usepackage{mathtools}
\usepackage{graphicx}
\graphicspath{ {./images/} }
\usepackage[hidelinks]{hyperref}
\usepackage{cleveref}
\usepackage{hyperref}
\usepackage{enumitem}
\usepackage{framed}
\usepackage{subcaption}
\usepackage{xspace}
\usepackage{thmtools} 
\usepackage{thm-restate}
\usepackage{tikz-feynman}
\usepackage{thmtools}
\usepackage{thm-restate}
\usepackage{ifthen} 
\usepackage{floatrow}

\usepackage{tikz}
\usepackage{mathdots}
\usepackage{xcolor}
\usepackage{diagbox}
\usepackage{colortbl}
\usepackage[absolute,overlay]{textpos}

\usetikzlibrary{shapes.misc}
\usetikzlibrary{decorations.pathmorphing}

\tikzset{snake it/.style={decorate, decoration=snake}}
\usetikzlibrary{math}

\usetikzlibrary{calc}
\usetikzlibrary{decorations.pathreplacing}
\usetikzlibrary{positioning,patterns}
\usetikzlibrary{arrows,shapes,positioning}
\usetikzlibrary{decorations.markings}

\tikzstyle{edge}=[very thick]
\definecolor{bostonuniversityred}{rgb}{0.8, 0.0, 0.0}
\definecolor{arsenic}{rgb}{0.23, 0.27, 0.29}
\tikzstyle{diredge}=[postaction={decorate,decoration={markings,
		mark=at position .95 with {\arrow[scale = 1]{stealth};}}}]
\tikzset{
    arrow/.style={decoration={markings, mark=at position 0.7 with
    {\fill(-0.09*#1,-0.03*#1) -- (0,0) -- (-0.09*#1,0.03*#1) -- cycle;}}, postaction={decorate}},
    arrow/.default=1
}
\tikzset{
    arow/.style={decoration={markings, mark=at position 1 with
    {\fill(-0.09*#1,-0.03*#1) -- (0,0) -- (-0.09*#1,0.03*#1) -- cycle;}}, postaction={decorate}},
    arow/.default=1
}
\tikzset{
    arrrow/.style={decoration={markings, mark=at position 0.9 with
    {\fill(-0.09*#1,-0.03*#1) -- (0,0) -- (-0.09*#1,0.03*#1) -- cycle;}}, postaction={decorate}},
    arow/.default=1
}

\newcommand{\fitellipsis}[2] 
{\draw [fill=white]let \p1=(#1), \p2=(#2), \n1={atan2(\y2-\y1,\x2-\x1)}, \n2={veclen(\y2-\y1,\x2-\x1)}
    in ($ (\p1)!0.5!(\p2) $) ellipse [ x radius=\n2/2+0cm, y radius=1.1cm, rotate=\n1];
}
\newcommand{\Fitellipsis}[2] 
{\draw [fill=white]let \p1=(#1), \p2=(#2), \n1={atan2(\y2-\y1,\x2-\x1)}, \n2={veclen(\y2-\y1,\x2-\x1)}
    in ($ (\p1)!0.5!(\p2) $) ellipse [ x radius=\n2/2+0cm, y radius=1.4cm, rotate=\n1];
}

\theoremstyle{plain}

\newtheorem*{thm*}{Theorem}
\newtheorem{thm}{Theorem}[section]
\Crefname{thm}{Theorem}{Theorems}

\newtheorem*{lem*}{Lemma}
\newtheorem{lem}[thm]{Lemma}
\Crefname{lem}{Lemma}{Lemmas}

\newtheorem*{claim*}{Claim}

\crefname{claim}{Claim}{Claims}
\Crefname{claim}{Claim}{Claims}

\newtheorem{prop}[thm]{Proposition}
\Crefname{prop}{Proposition}{Propositions}

\Crefname{remar}{Remark}{Remarks}

\crefname{cor}{Corollary}{Corollaries}

\newtheorem*{conj*}{Conjecture}

\crefname{conj}{Conjecture}{Conjectures}

\Crefname{qn}{Question}{Questions}

\newtheorem{obs}[thm]{Observation}
\Crefname{obs}{Observation}{Observations}

\Crefname{ex}{Example}{Examples}

\theoremstyle{definition}
\newtheorem{prob}[thm]{Problem}
\Crefname{prob}{Problem}{Problems}

\newtheorem{defn}[thm]{Definition}
\Crefname{defn}{Definition}{Definitions}

\theoremstyle{remark}

\renewenvironment{proof}[1][]{\begin{trivlist}
\item[\hspace{\labelsep}{\bf\noindent Proof#1.\/}] }{\qed\end{trivlist}}

\newcommand{\remove}[1]{}

\newcommand{\eps}{\varepsilon}

\title{\vspace{-0.85 cm}
Pancyclicity of Hamiltonian graphs}
\date{}

\author{
Nemanja Dragani\'c\thanks{
Department of Mathematics, ETH, Z\"urich, Switzerland. Research supported in part by SNSF grant 200021\_196965.
\newline
\emph{Emails}: \textbf{\{nemanja.draganic,david.munhacanascorreia, benjamin.sudakov\}@math.ethz.ch}.
}
\and
David Munh\'a Correia\footnotemark[1]
\and
Benny Sudakov\footnotemark[1]}
\begin{document} 
\maketitle
\begin{abstract}
An $n$-vertex graph is \emph{Hamiltonian} if it contains a cycle that covers all of its vertices, and it is \emph{pancyclic} if it contains cycles of all lengths from $3$ up to $n$. In 1972, Erd\H{o}s conjectured that every Hamiltonian graph with independence number at most $k$ and at least $n = \Omega(k^2)$ vertices is pancyclic. In this paper we prove this old conjecture in a strong form by showing that if such a graph has $n = (2+o(1))k^2$ vertices, it is already pancyclic, and this bound is asymptotically best possible.
\end{abstract}
\section{Introduction}

Hamiltonicity is one of the most central notions in graph theory, and it has been extensively studied by numerous researchers. The problem of deciding Hamiltonicity of a graph is NP-complete and therefore, a central theme in Combinatorics is to derive sufficient conditions for this property. The most classical one is Dirac’s theorem \cite{dirac1952some} which dates back to 1952 and states that every $n$-vertex graph with minimum degree at least $n/2$ contains a Hamilton cycle. Since then, many other interesting results about various aspects of Hamiltonicity have been obtained, see e.g. \cite{ajtai1985first,chvatal1972note,kuhn2013hamilton,krivelevich2011critical,krivelevich2014robust,MR3545109,ferber2018counting, cuckler2009hamiltonian, posa1976hamiltonian}, and the surveys \cite{gould2014recent, MR3727617}.

A related notion to Hamiltonicity is that of pancyclicity. An $n$-vertex graph is said to be \emph{pancyclic} if it contains all cycles of length from $3$ up to $n$. Trivially, pancyclicity is a stronger property than Hamiltonicity, and one might ask how much stronger it really is. In 1973, Bondy \cite{bondy10pancyclic} stated his celebrated meta-conjecture, indicating that the first property should be only slightly stronger than the latter. Indeed, he proposed that any non-trivial condition which implies that a graph is Hamiltonian should also imply that it is pancyclic (up to a certain collection of simple exceptional graphs). As an example, Bondy \cite{bondy1971pancyclic} himself first showed that every $n$-vertex graph with minimum degree at least $n/2$ is either pancyclic or isomorphic to the complete bipartite graph $K_{n/2,n/2}$, thus strengthening Dirac's theorem. His meta-conjecture sparked a lot of research which in turn has led to various appealing results and methods. For example, Bauer and Schmeichel \cite{bauer1990hamiltonian}, relying on previous results of Schmeichel and Hakimi \cite{schmeichel1988cycle}, showed that the sufficient conditions for Hamiltonicity given by Bondy \cite{bondy1980longest}, Chvátal \cite{chvatal1972hamilton} and Fan \cite{fan1984new} all imply pancyclicity, up to a certain small family of exceptional graphs. Furthermore, much like with Dirac's theorem, the classical result of Chvátal and Erdős \cite{chvatal1972hamilton} that a graph with connectivity number $\kappa(G)$ at least as large as its independence number $\alpha(G)$ is Hamiltonian, has also been addressed.
Namely, in 1990, Jackson and Ordaz \cite{jackson1990chvatal}, conjectured that if $\kappa(G) > \alpha(G)$, then $G$ must be pancyclic. An approximate form of this conjecture was proven by Keevash and Sudakov \cite{keevash2010pancyclicity}, who showed that $\kappa(G) \geq 600\alpha(G)$ is sufficient. The authors \cite{draganic2023chv} recently showed an asymptotically optimal result, proving that $\kappa(G) \geq (1+o(1))\alpha(G)$ is enough for pancyclicity. Finally, building on our work, Letzter \cite{letzter2023pancyclicity} proved the Jackson-Ordaz conjecture.

Bondy's meta-conjecture is about conditions for Hamiltonicity which imply pancyclicity. A natural and closely related question was first studied by Erd\H{o}s in the 1970s. Let $G$ be a Hamiltonian graph; under which assumptions can we guarantee that $G$ is also pancyclic or more generally, that it has many cycle lengths? One example of such a problem was suggested by Jacobson and Lehel at the 1999 conference “Paul Erd\H{o}s and His Mathematics”. They asked for the minimal number of cycle lengths in a $k$-regular $n$-vertex Hamiltonian graph. They conjectured (see Verstraëte \cite{verstraete2016extremal} for a stronger conjecture) that already when $k \geq 3$, there are $\Omega(n)$ many lengths. Improving on the previously best known lower bound of $\Omega(\sqrt{n})$ by Milans et al. \cite{milans2012cycle}, recently Buci\'c, Gishboliner and Sudakov \cite{bucic2021cycles} showed that any Hamiltonian graph with minimum degree at least $3$ has $n^{1-o(1)}$ different cycle lengths. 

As we already mentioned above, the earliest question of this flavor was studied by Erd\H{o}s about 50 years ago. In 1972 he posed the following problem in \cite{erdos1972some}. Given an $n$-vertex Hamiltonian graph with independence number $\alpha(G)\leq k$, how large does $n$ have to be in terms of $k$ in order to guarantee that $G$ is pancyclic? Erd\H{o}s \cite{erdos1972some} proved that it is enough to have $n=\Omega(k^4)$ and conjectured that already $n=\Omega(k^2)$ should be enough. A simple construction shows that this is best possible. Let $C_1,\ldots, C_k$ be disjoint
cliques of size $2k-2$, and let each $C_i$ have two distinguished vertices $a_i$ and $b_i$. Let $G$ be the graph obtained
by connecting $a_i$ and $b_{i+1}$ by an edge for each $i$ (taking addition modulo 
$k$). Notice that this graph has $n = 2k^2- 2k$ vertices, is Hamiltonian and its independence number is $k$. On the other hand, it is easy to check that it does not contain a cycle of length $2k - 1$, and thus it is not pancyclic. Indeed, observe that every cycle must be either a subgraph of
one of the cliques $C_i$, or contain all the vertices $a_i,b_i$ for each $i$. The first type of cycles all have
length at most $2k - 2$ and the latter have length at least $2k$.

In the last 50 years, there have been several improvements upon Erd\H{o}s's initial result. Firstly, Keevash and Sudakov \cite{keevash2010pancyclicity} proved that $n=\Omega(k^{3})$ vertices are enough to guarantee pancyclicity. Then, Lee and Sudakov \cite{lee2012hamiltonicity} improved this to $n=\Omega(k^{7/3})$, and more recently Dankovics \cite{dankovics2020low} showed that $n=\Omega(k^{11/5})$ vertices suffice. In this paper we completely resolve the conjecture of Erd\H{o}s in the following strong form.
\begin{thm}\label{thm:main}
Every Hamiltonian graph $G$ with $\alpha(G)\leq k$ and at least $2k^2+o(k^2)$ vertices is pancyclic.
\end{thm}
\noindent As shown by the previous construction, our result is tight up to the $o(k^2)$ error term. The rest of this paper is organized as follows.  In \Cref{sec:preliminarieas}, we state some well-known tools, and introduce some useful definitions. There we also prove the two key propositions that are used in the proof of Theorem~\ref{thm:main}, which is given in \Cref{sec:proof}. Finally, in \Cref{sec:concludingrem} we make some concluding remarks and mention some open questions; we also show how to get a short proof of the conjecture of Erd\H{o}s, that is, a proof of \Cref{thm:main} with a sufficiently large constant $C$ instead of the precise factor of 2.

\section{Preliminaries}\label{sec:preliminarieas}
\subsection{Notation and definitions}
We mostly use standard graph theoretic notation. Let $G$ be a finite graph. Denote by $V(G)$ its vertex set, and let $S_1,S_2\subseteq V(G)$. We denote by $G[S_1]$ the subgraph of $G$ induced by $S_1$, and by $E[S_1,S_2]$ the set of edges with one endpoint in $S_1$ and the other in $S_2$. Let $H$ be a subgraph of $G$. We denote by $G[H]$ the graph $G[V(H)]$.
A path $P=(x_0,x_1,\ldots,x_\ell)$ of length $\ell$ is a graph on vertex set $\{x_0,x_1,\ldots,x_\ell\}$ with an edge between $x_{i-1}$ and $x_{i}$ for all $i\in[\ell]$. We say that $x_0$ and $x_\ell$ are the endpoints of $P$, and we call $P$ an $x_0x_\ell$-path. 
If the vertices of the graph $G$ come with a given ordering, then we say that a path $P=(x_0,x_1,\ldots,x_\ell)$ contained in $G$ is increasing if $x_0<x_1<\ldots<x_\ell$.
We denote by $\alpha(G)$ the independence number of $G$. Given a digraph $D$, its independence number $\alpha(D)$ is defined as the independence number of the underlying graph. 

Given sets $A_1,A_2\subset \mathbb N$, we denote by $A_1+A_2$ the set of integers $c$ such that $c=a_1+a_2$ for some $a_1\in A_1$ and $a_2\in A_2.$ Throughout the paper we omit floor and ceiling signs for clarity of presentation, whenever it does not impact the argument.
\begin{defn}
Let  $a,b,p$ be positive real numbers. Given a graph $G$, and two vertices $x$ and $y$, we say that the pair $xy$ is $p$-dense in the interval $[a,b]$ if for every subinterval $[a',b']$ with $b'-a'\geq 
p$ there is an integer $\ell\in[a',b']$ and an $xy$-path in $G$ of length $\ell$.  
\end{defn}
\subsection{Standard tools}
\noindent Here we state and prove some standard facts which we use in our proof. We start with the following well-known result about directed graphs of Gallai and Milgram \cite{gallai1960verallgemeinerung}. A \emph{path cover} in a directed graph is a partition of its vertex set into directed paths, and its size is the number of such paths.
\begin{lem}[\cite{gallai1960verallgemeinerung}]\label{lem:partition}
Every directed graph $D$ has a path cover of size at most $\alpha(D)$.
\end{lem}
\noindent We also use the celebrated Ramsey's theorem.
\begin{thm}\label{lem:ramsey}
For every two positive integers $k,t$, there exists a large enough integer $n$, such that for any $k$-coloring of the edges of $K_n$, there is a monochromatic copy of $K_t$ in $K_n$.\end{thm}

\noindent The next lemma shows that we can partition a large proportion of the vertex set of a graph into sets with small radius, such that there are no edges between the parts.
\begin{lem}\label{lem:BFSpartition}
Let $G$ be an $n$-vertex graph and let $0<\gamma <\frac{1}{2}$. Then, there exists a collection of vertices $v_1, v_2, \ldots, v_r$ and disjoint sets $U_1, U_2, \ldots, U_r \subseteq V(G)$ such that the following hold.
\begin{enumerate}
    \item $v_j \in U_j$ for all $j$, and $\left| \bigcup_j U_j \right| \geq (1-\gamma)n$.
    \item Every vertex $u \in U_j$ has $\text{dist}(v_j,u) \leq \log_{1+\gamma} n$.
    \item There is no edge between any two sets $U_j, U_{j'}$ with $j \neq j'$.

\end{enumerate}
\end{lem}
\begin{proof}
We find the required sets and vertices with the following algorithm. We start with an arbitrary vertex $v_1 \in V(G)$ and consider the breadth-first-search tree rooted at $v_1$, that is, consider the sets $V^{(1)}_0,V^{(1)}_1,V^{(1)}_2, \ldots \subseteq V(G)$ defined as $V^{(1)}_i := \{u \in G : \text{dist}(v_1,u) = i\}$. Now, define $i_1 \geq 0$ to be the minimal $i$ such that $|V^{(1)}_{i+1}| \leq \gamma \left|V^{(1)}_0 \cup \ldots \cup V^{(1)}_i \right|$ and let $U_1 := V^{(1)}_0 \cup \ldots \cup V^{(1)}_{i_1}$. We now continue the process and do the same on the graph $G' := G \setminus \left(U_1 \cup V^{(1)}_{i_1 + 1}  \right)$. More precisely, take an arbitrary vertex $v_2 \in G'$ and consider again the breadth-first-search tree in $G'$ rooted at $v_2$. Like before, this gives an $i_2$ and sets $V^{(2)}_0, \ldots, V^{(2)}_{i_2+1}, U_2$. We then repeat this on the graph $G'' := G' \setminus \left(U_2 \cup V^{(2)}_{i_1 + 1}  \right)$ and continue doing this until we have no vertices left. By construction, the desired properties hold. Indeed, the only vertices not contained in $\bigcup_j U_j$ are in some $ V^{(j)}_{i_j+1}$, hence there are at most $\gamma n$ of them. For the second property, observe that each $U_j$ is of size at least $(1+\gamma)^{i_j}$, so 
$i_j\leq \log_{1+\gamma}n$. The third condition holds by construction, since we deleted all the neighbors of $U_i$ before defining $U_{i+1}$ .
\end{proof}
\noindent Finally, we state a trivial observation used throughout the proof of Theorem \ref{thm:main}. It will be used to state that appropriate combinations of internally vertex-disjoint paths of different lengths result in the construction of cycles of many different lengths.
\begin{obs}\label{obs:combining}
Let $G$ be a graph whose vertex set contains $t$ disjoint sets $S_1,\ldots, S_t$ and another set of $t$ vertices $v_1,\ldots,v_t$ outside of $\bigcup_{i=1}^t S_i$. For each $i\in[t]$, let $A_i\subset \mathbb N$ and suppose that for every $i$ the induced subgraph $G[v_i\cup S_i\cup v_{i+1}]$ is such that it contains a $v_iv_{i+1}$-path of length $\ell$ for each $\ell\in A_i$ (with $v_{t+1}=v_1$). Then for every $\ell\in A_1+\ldots+A_t$, the graph $G$ contains a cycle of length $\ell$.
\end{obs}
\subsection{Finding consecutive path lengths}
Next we show that in a graph with small independence number, we can find two vertices between which there exist paths of almost all `possible' lengths. We believe that this result is of independent interest and pose a problem related to it after its proof.
\begin{prop}\label{lem:pathlengthsininterval}
Let $G$ be an $n$-vertex graph with $\alpha(G) = k$ and let $0<\gamma <1/2$. Then there exist two vertices $u,v \in V(G)$ such that for every integer $\ell\in[\log_{1+\gamma}n, (1-\gamma)\frac{n}{k}]$ there is a $uv$-path of length $\ell$.
\end{prop}
\begin{proof}
First, we apply Lemma \ref{lem:BFSpartition} to $G$ to get vertices $v_i$ and sets $U_i$ for all $i\leq r$. Fix the graph $H := G[U_1 \cup \ldots \cup U_r]$ and let us orient the edges of $H$ in the following manner. For an edge $xy$ in $H$ with $x,y \in U_j$ (recall property \textit{3} of the sets $U_j$) orient it as $x \rightarrow y$ if $\text{dist}(v_j,x) < \text{dist}(v_j,y)$ and as $y \rightarrow x$ if $\text{dist}(v_j,x) > \text{dist}(v_j,y)$. In the case that $\text{dist}(v_j,x) = \text{dist}(v_j,y)$, orient the edge arbitrarily.
Now, since $\alpha(H) \leq \alpha(G) = k$ and $|H| \geq (1-\gamma)n$, by Lemma \ref{lem:partition} there must exist a directed path $\overrightarrow{P} = x_1 \rightarrow x_2 \rightarrow \ldots \rightarrow x_m$ in $H$ of length $m := \frac{|H|}{\alpha(H)} \geq (1-\gamma)\frac{n}{k}$. Let $j$ be such that $\overrightarrow{P} \subseteq U_j$, and for each $i\in[m]$ denote $d_i:=\text{dist}(v_j,x_i)$ and note that $d_i\leq \log_{1+\gamma}n$. 

Now we show that for all $\ell \in \left[d_m, m + d_1 \right] \supseteq \left[\log_{1+\gamma} n, m\right]$ there is path in $G$ of length $\ell$ between $v_j$ and $x_m$. For each $i\in[m]$ look at the $v_jx_m$-path $P_i$ obtained by concatenating the shortest $v_jx_i$-path with the path $x_{i+1}x_{i+2}\ldots x_m$. Since by definition, we have that $d_i \leq d_{\ell}$ for $i<\ell$, these two paths are vertex disjoint and their union is indeed a path as well. Moreover, for each $i$, we have that $|P_i|-1\leq|P_{i+1}|\leq |P_i|$ since $x_i \rightarrow x_{i+1}$ is an edge and thus, $d_i \leq d_{i+1} \leq d_i +1 $. Because $|P_1|=d_1+m$ and $|P_m|=d_m$, each path length in $\left[d_m, m + d_1 \right]$ is then attained by at least one of the constructed paths. Since $d_m\leq \log_{1+\gamma}n$ and $m \geq (1-\gamma)\frac{n}{k}$, this finishes the proof.
\end{proof}
Before moving on to the next section, it is worth noting that the above proposition is asymptotically tight. Indeed, an $n$-vertex graph $G$ with $\alpha(G) = k$ does not need to contain a path of length larger than $\frac{n}{k}$, as we can see from a disjoint union of cliques of size $n/k$.
Since the proof of this proposition uses a result about directed paths, we further ask if the following directed variant of it might be true as well.
\begin{prob}
Does Proposition \ref{lem:pathlengthsininterval} generalize to directed graphs?  If $G$ is a directed graph, how large an interval $I \subseteq [0,\frac{n}{k}]$ can we guarantee for which there are vertices $u,v$ with a directed $uv$-path of length $\ell$ for all $\ell \in I$?
\end{prob}
\subsection{Path shortening}

In this section we show that if a graph has small independence number and contains a long path $P$, then we can find a slightly shorter path $P'$ with the same endpoints, which satisfies certain additional properties. In a graph with independence number $k$, a path can clearly be shortened by considering $2k+1$ consecutive vertices on the path, and observing that there must be an edge (not contained in the path) between those vertices. This is the statement of the next simple lemma.

\begin{lem}\label{lem:easyjump}
Let $G$ be a graph with independence number $k$, and $P$ a path in $G$ with endpoints $x,y$ such that $|P| > 2k$. Then, there is an 
$xy$-path $P'$ with $V(P')\subseteq V(P)$ such that $|P| - 2k \leq |P'| < |P|$.
\end{lem}

\noindent The following proposition is one of the central results of this paper. It shows that we can shorten a path only by a little while also preserving a pre-specified set of vertices in the newly obtained path.

\begin{prop}\label{lem:jumpwithzigzag}
Let $G$ be an $n$-vertex graph with independence number $k$, and let $P$ be a path in $G$ with endpoints $x,y$. Further, let $c \in \mathbb{N}$ and $U \subseteq V(P)$, where $c \left(\frac{|P|-(4c-1)|U|}{2k}-1 \right)>k$. Then there is an 
$xy$-path $P'$ with the following properties.
\begin{enumerate}
    \item $U \subseteq V(P') \subseteq V(P)$.
    \item $|P| - (4c-3) \leq |P'| < |P|$.
\end{enumerate} 
\end{prop}
\noindent Before we give a proof of Proposition~\ref{lem:jumpwithzigzag}, we introduce some useful concepts. The key idea to prove the proposition is to find a certain structure in our graph which can be used to shorten a path in a graph with low independence number. The properties of this structure are captured by the notion of a \emph{special edge set}, defined below.

\begin{defn}\label{def:special}
Given a graph $H$ with ordered vertex set $(1,2,\ldots, n)$, we say that a sequence of vertices $(v_1,\ldots,v_{\ell})$ is \emph{special} if the following hold (see Figure \ref{fig:special set}).
\begin{itemize}
    \item  $v_{i+1} > v_{i}$ for all $i\in[\ell-1]$.
    \item $(v_{i},v_{i+1}+1)$ is an edge in $H$ for all $i\in [\ell-1]$. We call the set formed by those edges a \emph{special edge set}.
\end{itemize}
\end{defn}

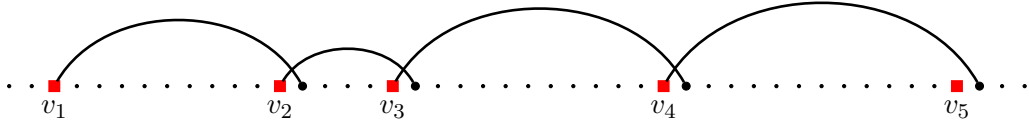
\begin{figure}[ht]
    \centering
    \begin{tikzpicture}[scale=1.5,main node/.style={circle,draw,color=black,fill=black,inner sep=0pt,minimum width=3pt}]
        \tikzset{cross/.style={cross out, draw=black, fill=none, minimum size=2*(#1-\pgflinewidth), inner sep=0pt, outer sep=0pt}, cross/.default={2pt}}
	\tikzset{rectangle/.append style={draw=brown, ultra thick, fill=red!30}}

	    \foreach \i in {1,...,46}
	    {
	        \node[main node, scale=0.4] (aux) at (\i*0.2-0.6,0){};
	    }
	    \node[rectangle, color=red, scale=0.4] (a) at (0,0) [label=below:$v_1$]{};
	    \node[rectangle, color=red, scale=0.4] (a1) at (2,0)[label=below:$v_2$]{};
	    \node[rectangle, color=red, scale=0.4] (a2) at (3,0)[label=below:$v_3$]{};
	    \node[rectangle, color=red, scale=0.4] (a3) at (5.4,0)[label=below:$v_4$]{};
	     \node[rectangle, color=red, scale=0.4] (a4) at (8,0)[label=below:$v_5$]{};
	    \node[main node] (a4) at (8.2,0){};
	    
	    \node[main node] (b1) at (2.2,0){};
	    \node[main node] (b2) at (3.2,0){};
	    \node[main node] (b3) at (5.6,0){};
	    
	    \draw[line width= 1 pt] (a) to [bend left=60](b1);
	    \draw[line width= 1 pt] (a1) to [bend left=60](b2);
	    \draw[line width= 1 pt] (a2) to [bend left=60](b3);
	    \draw[line width= 1 pt] (a3) to [bend left=60](a4);

    \end{tikzpicture}
    \caption{An illustration of a special edge set. The special vertices are colored red.}
    \label{fig:special set}
\end{figure}
\noindent We can now find special vertex sequences in graphs in the following manner.
\begin{lem}\label{lem:zigzag}
Let $H$ be a graph on the vertex set $[n]$, let $U' \subseteq [n]$ and suppose $H$ has independence number $k$. Then, there exist a special vertex sequence with at least $\frac{n-|U'|}{2k}$ vertices in $[n]\setminus U'$.
\end{lem}
\begin{proof}
Let $H'$ be the graph obtained from $H$ by removing all edges of the form $\{i,i+1\}$. Since the removed edges form a graph with chromatic number at most 2, we must have that $\alpha (H') \leq 2 \alpha(G) \leq 2k$. Let us also direct the edges of $H'$ so that an edge $ij$ is oriented from $i$ to $j$ if $i < j$. Applying Lemma~\ref{lem:partition} to this directed graph, we obtain a partition of $[n]$ into at most $2k$ directed paths. Denote the digraph formed by the union of those paths as $F$. An important property of $F$ that we are going to use is that all outdegrees and indegreees of vertices in $F$ are at most one. 

Having obtained this path cover, we want to find another decomposition but now of the edges of $F$, and into a small number of special edge sets. Simply take $\mathcal{M}$ to be a smallest collection of edge disjoint special edge sets which decompose the edges of $F$ (this exists as one such collection is the set of all single edges in $F$). If $(v_1,v_2,\ldots,v_\ell)$ is the special sequence corresponding to a special edge set $M$ in $\mathcal{M}$, then note that $v_\ell$ must be a vertex of out-degree $0$ in $F$. Indeed, if this is not the case, then there exists $u > v_\ell$ such that $v_\ell \rightarrow u$ is an edge of $F$. Let $M' \in \mathcal{M}$ be the special edge set which contains it and note that then $v_\ell$ is the first vertex of this special edge set, since otherwise it would contain the edge $\{v_{\ell -1}, v_{\ell}+1\} \in M$ and contradict the edge-disjointness of the special edge sets in $\mathcal{M}$. But now we can 'concatenate' $M$ and $M'$ to form a larger special edge set $M \cup M'$, which contradicts the minimality of $\mathcal{M}$. 

Notice that the special edge sets in $\mathcal{M}$ have disjoint corresponding special vertex sequences. Indeed, since these special edge sets are non-empty and edge-disjoint, the only possibility for a common vertex in two different special vertex sequences would be if some vertex $v$ is the first vertex of one sequence and the last vertex of another one. In turn, as shown in the previous paragraph, this would contradict the maximality of $\mathcal{M}$. Now, a consequence of having disjoint corresponding special vertex sequences in $\mathcal{M}$ and the previous paragraph is that each one of these has a unique vertex of out-degree $0$. Moreover, let $S \subseteq V(H)$ denote the set of vertices which do not belong to any special vertex sequence of $\mathcal{M}$ and notice that all $v \in S$ also have out-degree $0$ in $F$. In turn, since $F$ is a decomposition of $[n]$ with at most $2k$ paths, there are at most $2k$ vertices of outdegree zero and thus $|\mathcal{M}| + |S| \leq 2k$. Hence, there exists one special vertex sequence in $\mathcal{M} \cup S$ (allowing a single vertex to be a special vertex sequence) with at least $\frac{n - |U'|}{|\mathcal{M}|+|S|} \geq \frac{n-|U'|}{2k}$ vertices not in $U'$, which finishes the proof.
\end{proof}

\noindent We can now prove our proposition.
\begin{proof}[ of Proposition \ref{lem:jumpwithzigzag}]
Let $(1,2,\ldots,|P|)$ be an ordering of the vertices of $P$, with the endpoints $x=1$ and $y=|P|$. 
Let $U_c$ be the set of vertices $v$ in $P$ such that there is a $u\in U$ with $|v-u|\leq 2c-1$. By applying Lemma~\ref{lem:zigzag} to $G[V(P)]$ and the set $U_c$, we obtain a special vertex sequence $(v_1,\ldots,v_\ell)$ corresponding to a special edge set $M$, with at least $\frac{|P|-|U_c|}{2k}$ vertices outside of $U_c$. First, notice that for each $v_{i+1}\notin U_c$ we may assume that $v_{i+1}-v_i\geq 2c$. Indeed, if that was not the case, then we obtain the desired path $P'$ from $P$ by replacing the interval $[v_i,v_{i+1}+1]$ by the edge $(v_i,v_{i+1}+1)$, noting that there are no vertices from $U$ inside of the removed interval, since otherwise $v_{i+1}$ would be in $U_c$, a contradiction.

Now we define, for each special vertex $v = v_i$ which is not in $U_c\cup \{v_1\}$, the $c$-element set $S_v:=\{v-1, v-3,\ldots, v-(2c-1)\}$ contained in the $(2c-1)$-element interval $I_v:=[v-(2c-1),v-1]$, which is disjoint to $U$. Since $v_{i}-v_{i-1}\geq 2c$, all of these sets are disjoint and further, disjoint to $U$. Now, the union $S$ of those sets is of size at least $\left(\frac{|P|-|U_c|}{2k}-1 \right)c$. Therefore, since $|U_c|\leq (4c-1)|U|$, we get that $|S|\geq k+1$ by our assumption on $c$. Hence, there exists an edge $e$ in $G$ spanned by $S$, since the independence number of $G$ is $k$. We may assume that this edge does not lie inside some $S_v$, as otherwise we can again get the desired path $P'$ by using $e$ instead of the interval which it bridges, avoiding at least one, but at most $2c-1$ vertices which are not in $U$. Hence, $e=ab$ with $a<b$ is between two distinct sets $S_{v_i}$ and $S_{v_j}$. Now we can find the required path $P'$ as shown in Figure~\ref{fig:n-1}, avoiding at most $4c-3$ vertices.

\noindent More precisely, we obtain the required path $P'$ as the union of the following paths:
\begin{itemize}
    \item The part of the path $P$ which connects $x$ to $a$, plus the edge $(a,b)$.
    \item The increasing path $P_2$ obtained by the following iterative procedure. First, initialize $P_2$ to be the edge $(v_i,v_i+1)$. Repeat the following. Let $r$ be the last vertex of $P_2$. If $r$ is a special vertex $r=v_t$, then add the edge $(r,v_{t+1}+1)$ to $P_2$, and update $r=v_{t+1}+1$. If $r$ is not a special vertex, add the edge $(r,r+1)$ to $P_2$, and update $r=r+1$. We stop when either $r=b$ or $r=v_j+1$.
    
    \item The increasing path $P_3$ obtained by the following iterative procedure. First we initialize $P_3$ to be the edge $(v_i,v_{i+1}+1)$. Repeat the following (exactly as for the previous path). Let $r$ be the last vertex of $P_3$. If $r$ is a special vertex $r=v_t$, then add the edge $(r,v_{t+1}+1)$ to $P_3$, and update $r=v_{t+1}+1$. If $r$ is not a special vertex, add the edge $(r,r+1)$ to $P_3$, and update $r=r+1$. We stop when either $r=b$ or $r=v_j+1$.
    \item The part of the path $P$ which connects $v_j+1$ to $y$.
\end{itemize}
It is easy to see that this path contains only vertices of $P$ and that it does not contain the vertex $v_j$. Furthermore, the only other vertices from $P$ which the new path $P'$ avoids are the vertices in $I_{v_i}$ which are strictly larger than $a$, and the vertices in $I_{v_j}$ which are strictly larger than $b$. So in total, the new path $P'$ avoids at most $1+(|I_{v_i}|-1)+(|I_{v_j}|-1)=4c-3$ vertices of $P$, which completes the proof.
\end{proof}

\begin{figure}[ht]
    \centering
    \begin{tikzpicture}[scale=1,main node/.style={circle,draw,color=black,fill=black,inner sep=0pt,minimum width=7pt}]
    
    	\tikzset{rectangle/.append style={draw=brown, ultra thick, fill=red!30}} 
        
      \node[rectangle, color=black, scale = 0.3] (s1) at (0,0) {$x$}; 
      \node[rectangle, color=black, scale = 0.3] (s2) at (17,0) {$y$};
    
    \node[rectangle, scale=0.6, color=red, opacity=1] (a1) at (1,0){};
    \node[rectangle, scale=0.6, color=red, opacity=1] (a2) at (4,0){};
    \node[rectangle, scale=0.6, color=red, opacity=1] (a3) at (6,0){};
    \node[rectangle, scale=0.6, color=red, opacity=1] (a4) at (9,0){};
    \node[rectangle, scale=0.6, color=red, opacity=1] (a5) at (11,0){};
    \node[rectangle, scale=0.6, color=red, opacity=1] (a6) at (13,0){};
    \node[rectangle, scale=0.6, color=red, opacity=1] (a7) at (16,0){};
    
    \node[main node, scale=0.6, color=black, opacity=1] (b0) at (1.3,0){};
    \node[main node, scale=0.6, color=black, opacity=1] (b1) at (4.3,0){};
    \node[main node, scale=0.6, color=black, opacity=1] (b2) at (6.3,0){};
    \node[main node, scale=0.6, color=black, opacity=1] (b3) at (9.3,0){};
    \node[main node, scale=0.6, color=black, opacity=1] (b4) at (11.3,0){};
    \node[main node, scale=0.6, color=black, opacity=1] (b5) at (13.3,0){};
    \node[main node, scale=0.6, color=black, opacity=1] (b6) at (16.3,0){};
    
     \node[main node, scale=0.6, fill=white, opacity=1] (p1) at (3.7,0){};
      \node[main node, scale=0.6, fill=white, opacity=1] (p11) at (3.4,0){};
    \node[main node, scale=0.6, fill=white, opacity=1] (p111) at (3.1,0){};
    \node[main node, scale=0.6, fill=white, opacity=1] (p3) at (8.7,0){};
    \node[main node, scale=0.6, fill=white, opacity=1] (p4) at (8.4,0){};
    \node[main node, scale=0.6, fill=white, opacity=1] (p5) at (8.1,0){};
    \node[main node, scale=0.6, fill=white, opacity=1] (p6) at (12.7,0){};
    \node[main node, scale=0.6, fill=white, opacity=1] (p6) at (12.4,0){};
    \node[main node, scale=0.6, fill=white, opacity=1] (p6) at (12.1,0){};
    \node[main node, scale=0.6, fill=white, opacity=1] (p7) at (15.7,0){};
    \node[main node, scale=0.6, fill=white, opacity=1] (p5) at (15.4,0){};
    \node[main node, scale=0.6, fill=white, opacity=1] (p9) at (15.1,0){};

     \draw[line width= 1 pt] (a1) to [bend left=50](b1);
  \foreach \i in {2,...,6}
     {
     \draw[line width= 2 pt,  color = blue] (a\i) to [bend left=50](b\i);
     }
       \draw[line width= 2 pt, color = blue] (p11) to [bend right=28](p9); 
       
       \draw[line width= 2 pt,  color = blue] (p11) to (s1);
       \draw[line width= 2 pt,  color = blue] (a2) to (a3);
       \draw[line width= 2 pt,  color = blue] (b2) to (a4);
       \draw[line width= 2 pt,  color = blue] (b3) to (a5);
       \draw[line width= 2 pt,  color = blue] (b4) to (a6);
       \draw[line width= 2 pt,  color = blue] (b5) to (p9);
       \draw[line width= 2 pt,  color = blue] (b6) to (s2);
       
    \node[rectangle, scale=0.6, color=red, opacity=1] (a1) at (1,0){};
    \node[rectangle, scale=0.6, color=red, opacity=1] (a2) at (4,0){};
    \node[rectangle, scale=0.6, color=red, opacity=1] (a3) at (6,0){};
    \node[rectangle, scale=0.6, color=red, opacity=1] (a4) at (9,0){};
    \node[rectangle, scale=0.6, color=red, opacity=1] (a5) at (11,0){};
    \node[rectangle, scale=0.6, color=red, opacity=1] (a6) at (13,0){};
    \node[rectangle, scale=0.6, color=red, opacity=1] (a7) at (16,0){};
    
    \node[main node, scale=0.6, color=black, opacity=1] (b0) at (1.3,0){};
    \node[main node, scale=0.6, color=black, opacity=1] (b1) at (4.3,0){};
    \node[main node, scale=0.6, color=black, opacity=1] (b2) at (6.3,0){};
    \node[main node, scale=0.6, color=black, opacity=1] (b3) at (9.3,0){};
    \node[main node, scale=0.6, color=black, opacity=1] (b4) at (11.3,0){};
    \node[main node, scale=0.6, color=black, opacity=1] (b5) at (13.3,0){};
    \node[main node, scale=0.6, color=black, opacity=1] (b6) at (16.3,0){};
    
     \node[main node, scale=0.6, fill=white, opacity=1] (p1) at (3.7,0){};
      \node[main node, scale=0.6, fill=white, opacity=1] (p11) at (3.4,0){};
    \node[main node, scale=0.6, fill=white, opacity=1] (p111) at (3.1,0){};
    \node[main node, scale=0.6, fill=white, opacity=1] (p3) at (8.7,0){};
    \node[main node, scale=0.6, fill=white, opacity=1] (p4) at (8.4,0){};
    \node[main node, scale=0.6, fill=white, opacity=1] (p5) at (8.1,0){};
    \node[main node, scale=0.6, fill=white, opacity=1] (p6) at (12.7,0){};
    \node[main node, scale=0.6, fill=white, opacity=1] (p6) at (12.4,0){};
    \node[main node, scale=0.6, fill=white, opacity=1] (p6) at (12.1,0){};
    \node[main node, scale=0.6, fill=white, opacity=1] (p7) at (15.7,0){};
    \node[main node, scale=0.6, fill=white, opacity=1] (p5) at (15.4,0){};
    \node[main node, scale=0.6, fill=white, opacity=1] (p9) at (15.1,0){};
       
    \node[scale=1] (a) at  (0, -0.3){$x$};
    \node[scale=1] (a) at  (17,-0.3){$y$};
    
    \node[scale=1] (a) at  (3.2,0.4){$S_{v_i}$};
    \node[scale=1] (a) at  (15.4,0.3){$S_{v_j}$};
    
     \node[scale=1] (a) at  (3.4,-0.2){$a$};
    \node[scale=1] (a) at  (15.1,-0.2){$b$};

    \end{tikzpicture} 
    \caption{
    The shorter path $P'$ is drawn in blue, the special vertices are represented with red squares, while the vertices after them on $P$ are represented with black dots. The vertices in the sets $S_v$ are represented with white dots. Also note the following about the black dots and red squares which are between $a$ and $b$ in $P$. It could be that one of these black dots and the red square which is after it in our drawing, are the same vertex. However, the construction of the path $P'$ remains the same (for example, the fourth black dot and the fifth red square could be the same vertex). Notice also in the drawing that some of the red squares do not have a corresponding set of white dots. These are precisely those special vertices which are in $U_c \cup \{v_1\}$.
    }
    
    \label{fig:n-1}
\end{figure}

\section{Proof of Theorem \ref{thm:main}}\label{sec:proof}
Let $\varepsilon>0$ be a small enough constant, and let $k$ be sufficiently large in terms of $\eps$. Let $G$ be a Hamiltonian graph with $\alpha(G)\leq k$ on $n \geq (2+\varepsilon)k^2$ vertices. Our goal is to prove that $G$ is pancyclic. It will be convenient for us to consider different ranges of cycle lengths, and for each range we have a separate subsection which deals with it.
\subsection{Lower range: from $3$ to $(2+\varepsilon)k$}
Showing that $G$ contains all cycles of lengths between $3$ and $(2+\varepsilon)k$ only requires the fact that $G$ has no independent set of size $k+1$. Indeed, this boils down to the study of cycle-complete Ramsey numbers. Namely, the cycle-complete Ramsey number $r(C_\ell,K_s)$ is the smallest number $N$ such that every graph on $N$ vertices either contains a copy of $C_\ell$ or an independent set of size $s$. 
The following result of Erd\H{o}s, Faudree, Rousseau and Schelp \cite{erdos1978cycle}, along with a more recent result by Keevash, Long and Skokan \cite{keevash2021cycle} cover the mentioned range of cycle lengths we need.

\begin{thm}[\cite{erdos1978cycle}]\label{thm:erdos cycle-complete}
Let $\ell\geq 3$ and $s\geq 2$. Then $r(C_\ell, K_s) \leq\left((\ell-2)(s^{1/x}+2)+1\right)(s-1)$, where $x=\lfloor \frac{\ell-1}{2}\rfloor$.
\end{thm}

The next result by Keevash, Long and Skokan gives the precise behaviour of cycle-complete Ramsey numbers in a wide range of parameters, and proves a conjecture from \cite{erdos1978cycle}.

\begin{thm}[\cite{keevash2021cycle}]\label{precise cycle-complete}
There exists $C \geq 1$ so that $r(C_\ell , K_s) = (\ell - 1)(s- 1) + 1$ for $s \geq 3$ and $\ell\geq C \frac{\log s}{\log\log s}$.
\end{thm}

Now, note that since $G$ contains no independent set of size $k+1$, Theorem~\ref{thm:erdos cycle-complete} implies the existence of a cycle of length $\ell$ for every $\ell\in[3,\log k]$, while Theorem~\ref{precise cycle-complete} covers the range of $[\log k,(2+\varepsilon)k]$.

\subsection{Upper range: from $\frac{1000}{\varepsilon^2}k$ to $n$}\label{sec:upperrange}
First, note that all cycle lengths in $[2k^2+2k,n]$ can be obtained by iteratively applying Proposition~\ref{lem:jumpwithzigzag} with $c=1$, and $U=\emptyset$, and always shortening the cycle by one. Indeed, denote by $P$ any Hamilton path contained in a Hamilton cycle in $G$, and let $x$ and $y$ be its endpoints. As long as $n> 2k^2+2k$, by applying the mentioned proposition we get a path which is by $c=1$ shorter than $P$ and has the same endpoints, so adding the edge $xy$ to it creates a cycle of length $n-1$. We remove the remaining vertex and repeat. This gives all cycle lengths in $[2k^2+2k,n]$.

Now we turn to the cycle lengths in $\left[\frac{1000}{\varepsilon^2}k,2k^2+2k\right]$. For this we need the following lemma.

\begin{lem}\label{lem:partitionintomatchingcycle}
Let $G$ be a Hamiltonian graph on $n$ vertices with independence number $k$. Then, there is a partition of the vertices of $G$ into a cycle $C$ and a set $S$ of size $|S|= \frac{\varepsilon n}{20}$, such that there is a matching $M \subseteq E[C,S]$ which covers $S$.
\end{lem}
\begin{proof}
To show this claim, we apply Proposition~\ref{lem:jumpwithzigzag} iteratively $\varepsilon n/20$ times as follows. We always have $c=1$, and in the beginning we set $U_M=S=\emptyset$, and we set $M$ to be an empty matching and $C$ a Hamilton cycle in $G$.
During the procedure $C$ and $S$ are always disjoint and partition the vertices of $G$, $M$ is a matching in $E[C,S]$ which covers $S$, and $U_M$ is the set of endpoints of $M$ in $C$.

In the first step we apply Proposition~\ref{lem:jumpwithzigzag} with the mentioned values of $c=1$ and $U=\emptyset$, to get a cycle $C'$ of length $n-1$, and a vertex $v$ which is not on $C'$. Let $v'$ be a neighbor of $v$ in the cycle $C$. Now we set $C=C'$, $S = \{v\}$, $M=\{vv'\}$ and $U_M=\{v'\}$.

In the $i$-th step of the procedure, we let $U$ denote the set of vertices which are either in $U_M$, or adjacent to a vertex in $U_M$ on the current cycle $C$.
We apply Proposition~\ref{lem:jumpwithzigzag} with $c=1$ and $U$ to the graph $G[C]$, to get a cycle $C'$ of length $|C|-1$ and a vertex $v \in C \setminus C'$ which is not in $U$. Again, we denote by $v'$ the neighbor of $v$ in $C$, and we set $M:=M\cup \{vv'\}$, $S := S \cup \{v\}$, $U_M=U_M \cup\{v'\}$ and $C=C'$. Notice that since we only perform $\varepsilon n/20$ steps, at each point we have that $|U_M|\leq \varepsilon n/20$, and that $|C|\geq n-\varepsilon n/20$. Together with the fact that $|U|\leq 3|U_M|$, this gives that $\frac{|C|-3|U|}{2k}-1> k$, so we can always successfully apply Proposition~\ref{lem:jumpwithzigzag} with $c=1$. The resulting matching is then of size $\frac{\varepsilon n}{20}$, and evidently satisfies the given requirements.

\end{proof}

We are ready to show how to get the cycle lengths in $\left[\frac{1000}{\varepsilon^2}k,2k^2+2k\right]$. We first apply the above lemma to get a cycle $C$ of length $n-\frac{\varepsilon n}{20}\geq 2k^2+\frac{2\varepsilon k^2}{3}$ and outside of it a set $S$ of size $|S|=\frac{\varepsilon n}{20}$, together with a matching $M$ between them which covers $S$. Split the cycle $C$ into $4/\varepsilon$ intervals of (almost) equal size. By pigeonholing, at least one of those intervals contains at least $\frac{\varepsilon^2n}{80}$ endpoints of $M$. Let $S'$ be the subset of $S$ of vertices corresponding to those endpoints in $M$. We apply Lemma~\ref{lem:pathlengthsininterval} to the graph $G[S']$ with say $\gamma=1/100$, and conclude that there are two vertices $x'$ and $y'$ in $G[S']$, between which there exists a path of length $\ell$ in $G[S']$, for every $\ell\in [\frac{\varepsilon^2k}{100}, \frac{\varepsilon^2k}{50}]$.

Let $x$ and $y$ be the vertices in $C$ corresponding to $x'$ and $y'$ in $M$. Let $P$ be the longer path in $C$ which connects $x$ and $y$, so that by the choice of $S'$ we have $|P|\geq |C|-\frac{\varepsilon n}{4}>2k^2+2k$. 
Now we use Proposition~\ref{lem:jumpwithzigzag} to show that in $G[P]$ the pair $(x,y)$ is $\frac{\varepsilon^2 k}{100}$-dense in $\left[\frac{900k}{\varepsilon^2},2k^2+2k\right]$. Denote $P_0:=P$, and we obtain the path $P_i$ from path $P_{i-1}$ by applying Proposition~\ref{lem:jumpwithzigzag} with $c=\frac{\varepsilon^2k}{400}$ and $U=\emptyset$. We do this until $|P_i|<\frac{900k}{\varepsilon^2}$ and then we stop our procedure. Notice we could do each step of the procedure as we always had $c \left(\frac{|P_i|}{2k}-1 \right)>k$. Hence we obtain a sequence of $xy$-paths of decreasing lengths, where $|P_i|\geq|P_{i-1}|-4c+3\geq |P_{i-1}|-\frac{\varepsilon^2k}{100}$. Since $|P_0|>2k^2+2k$ and the last path is of length at most $\frac{900k}{\varepsilon^2}$, we indeed get that in $G[P]$ the pair $(x,y)$ is $\frac{\varepsilon^2 k}{100}$-dense in $\left[\frac{900k}{\varepsilon^2},2k^2+2k\right]$.

Now, applying \Cref{obs:combining} to the graph $G[P\cup S']$ with $v_1=x$ and $v_2=y$, gives all cycle lengths in $\left[\frac{1000k}{\varepsilon^2},2k^2+2k\right]$.

\subsection{Middle range : from $(2+\eps) k$ to $\frac{1000}{\varepsilon^2}k$}\label{sec:middlerange}

By Lemma \ref{lem:partitionintomatchingcycle}, there exists a partition of the vertices of $G$ into a cycle $C$ and a set $S$ such that $|S| = \eps n/20$, along with a matching $M \subseteq E[C,S]$ which covers $S$. Denote the vertices along the cycle $C$ with $(1,\ldots, N)$. For each vertex $x \in S$, we denote by $m(x)$ the vertex in $C$ matched to $x$ in $M$. We now remove from $S$ the at most $\frac{1000k}{\varepsilon^2}$ vertices $x$ which have $m(x)\in \{1,2,\ldots,\frac{1000k}{\varepsilon^2}\}$. Hence, there are at least $\varepsilon n/22$ vertices remaining in $S$.

We let $B_1:=S$ and apply Proposition \ref{lem:pathlengthsininterval} with $\gamma=1/100$ to the graph $G[B_1]$, to find
a pair of vertices $x_1,y_1$ such that for all $\ell \in \left[\frac{\varepsilon k}{100}, \frac{\varepsilon k}{50}\right]$, there is an $x_1y_1$ path of length $\ell$ in $G[S]$; we set $B_2:=B_1-\{x_1,y_1\}$.
We repeat this $t= \varepsilon k^2/40$ times, i.e., we apply Proposition~\ref{lem:pathlengthsininterval} to $B_i$ to obtain vertices $x_i,y_i$ such that for all $\ell \in \left[\frac{\varepsilon k}{100}, \frac{\varepsilon k}{50}\right]$, there is an $x_iy_i$-path of length $\ell$ in $G[B_i]\subseteq G[S]$, and we set $B_{i+1}:=B_i-\{x_i,y_i\}$.
Note that each $B_i$ is of size at least $|S|-2t\geq \frac{\varepsilon n}{22} - 2t \geq \frac{\varepsilon k^2}{11} - \frac{\varepsilon k^2}{20} \geq \frac{\varepsilon k^2}{30}$, so we always can successfully apply Proposition~\ref{lem:pathlengthsininterval}.

Before we make our first crucial observation, we introduce some notation. First, with possible relabeling, let us suppose that for each $i$, we have that $m(x_i) < m(y_i)$. Now, for each $i$, let $P_i$ denote the subpath $\left(m(x_i)-\frac{1000k}{\varepsilon^2}, m(x_i)-\frac{1000k}{\varepsilon^2}-1,\ldots, m(x_i) -1,m(x_i)\right)$ of length $\frac{1000k}{\varepsilon^2}$. Notice that, since $m(x_i)>\frac{1000k}{\varepsilon^2}$, none of those paths contains vertex $1$.
\begin{lem}
\label{induced}
If there is an $i$ such that $G[P_i]$ does not contain an increasing path of length $\varepsilon^3k$ as an induced subgraph with $m(x_i)$ as an endpoint \footnote{Recall that a path $P$ is an induced subgraph of $G[P_i]$ if its vertices belong to $V(P_i)$ and except for the edges of $P$, $G[P]$ does not contain any other edges.}, then $G$ contains all cycle lengths in $\left[(2+\varepsilon)k,\frac{1000k}{\varepsilon^2}\right]$.
\end{lem} 
\begin{proof}
Assume for sake of contradiction that $G[P_i]$ does not contain such a path. Then, the following holds.
\begin{claim*}
In the graph $G[P_i]$, the endpoints of $P_i$ are $\varepsilon^3k$-dense in $\left[0,\frac{1000k}{\varepsilon^2}\right]$.
\end{claim*}

\begin{proof}[ of Claim]
Consider the following procedure. We begin with $P_i$; by assumption, there exists a chordal edge among the last $\eps^3 k$ vertices of $P_i$ ending with $m(x_i)$. Otherwise, these vertices would induce an increasing path in $G[P_i]$.
Thus we obtain a path $P_i'$ by adding this edge to $P_i$ instead of the interval between the endpoints of this edge. We repeat this procedure, each time finding a chordal edge in the newly obtained path, and we do this until our path has length at most $\varepsilon^3k$. Hence, we obtain a sequence of paths such that two consecutive paths lengths are at most $\varepsilon^3k$ apart, while the last path has length at most $\varepsilon^3k$. Since the endpoints always remain the same, this implies the statement of the claim.
\end{proof}

\noindent Consider now the path $P$ contained in the cycle $C$, and which is spanned by the vertices in the interval
$\left[m(y_i),N]\cup[1,m(x_i)-\frac{1000k}{\varepsilon^2}\right]$. By Lemma \ref{lem:easyjump}, there exists a path $P'$ with the following properties: $V(P') \subseteq V(P)$, the endpoints of $P'$ are the same as those of $P$, and $|P'| \leq 2k$. 

Then, in order to finish, recall that $x_i,y_i$ are such that for all $\ell \in \left[\frac{\varepsilon k}{100}, \frac{\varepsilon k}{50}\right]$, there exists an $x_iy_i$-path of length $\ell$ in $G[S]$. Further, by the above claim we got that in the graph $G[P_i]$, the endpoints of $P_i$ are $\varepsilon^3k$-dense in $\left[0,\frac{1000k}{\varepsilon^2}\right]$.
Hence we can use \Cref{obs:combining} on the graph $G$ with $v_1=m(x_i), v_2=m(x_i)-\frac{1000k}{\varepsilon^2}$ and $v_3=m(y_i)$, where $S_1$ are the internal vertices of $P_i$, $S_2$ the internal vertices of $P$ and $S_3=S$ to obtain all cycle lengths in $\left[(2+\eps)k,\frac{1000}{\varepsilon^2}k\right]$.
\end{proof}

We have $t=\varepsilon k^2/40$ paths $P_i$ of length $1000k/\varepsilon^2$ such that each of them corresponds to an interval of vertices in $C$ and therefore intersects at most $2000k/\varepsilon^2$ other such paths. Thus, we can choose a collection of $r=\frac{t}{2000k/\varepsilon^2+1}\geq \sqrt{k}$  of those paths which are all disjoint.  With possible renaming, w.l.o.g.\ we may assume that those paths are $P_1,\ldots,P_r$. 
Using Lemma \ref{induced}, we may also assume that there are induced increasing subpaths $Q_1, \dots,Q_r$ of $G[P_1],\ldots,G[P_r]$ with endpoints $m(x_1),\ldots,m(x_r)$ respectively and of length $\varepsilon^3 k$. 

Let us now define an auxiliary colored complete graph $H$ on $[r]$ in the following manner. For each $i \in [r]$, partition $Q_i$ into three consecutive subpaths $Q^{3}_i,Q^{2}_i,Q^{1}_i$ of size $|Q_i|/3=\eps^3 k/3$, with $Q^{1}_i$ containing $m(x_i)$. Now, for $i,j \in [r]$, we color the edge $ij$ in $H$ \emph{red} if in $G$ both $E[Q^{1}_i,Q^{1}_j]$ and $ E[Q^{3}_i,Q^{3}_j]$ are non-empty. We color it \emph{blue} if $E[Q^{1}_i,Q^{1}_j] = \emptyset$, and in the remaining case, we color it \emph{green}. 
\begin{claim*}
There are no blue or green cliques in $H$ of size larger than $6/\eps^3$.
\end{claim*}
\begin{proof}
Suppose there exists a blue clique $\{i_1, \ldots, i_\ell\}$ in $H$. Since each $Q^{1}_{i_j}$ is an induced path, its 
odd vertices form an independent set of size $|Q^{1}_{i_j} |/2$. Moreover, by assumption, there are no edges between two $Q^{1}_{i_j}$'s and therefore the set $\bigcup_{1 \leq j \leq \ell} V(Q^{1}_{i_j})$ must contain an independent set of size at least $$\sum_{1 \leq j \leq \ell} \frac{|Q^{1}_{i_j}|}{2} \geq \ell \cdot \left(\eps^3 k/6 \right) .$$ Since $\alpha(G) \leq k$, this implies that $\ell \leq 6/\eps^3$. An analogous argument deals with green cliques.
\end{proof}
Given the above claim and Theorem \ref{lem:ramsey}, and since $H$ has $r>\sqrt{k}$ vertices where we chose $k$ large enough in terms of $\varepsilon$, we have that there exists a red clique in $H$ of size at least $4\varepsilon^{-7}$. 

Denote by $I$ the vertices/indices contained in this clique, so that for all $i,j\in I$ we have that there is an edge $e^1_{ij}$  between $Q^{1}_i$ and $Q^{1}_j$ and an edge $e^3_{ij}$  between $ Q^{3}_i$ and $Q^{3}_j$. For simplicity of notation, w.l.o.g.\ we may assume that the indices in $I$ are $\{1, 2, 3, \ldots, |I|\}$ according to the ordering of the vertices $m(x_i)$ for $i \in I$ - that is, we now have $m(x_1) < m(x_2) < \ldots < m(x_{|I|})$. Denote by $z$ the endpoint of $Q_1$ which is not $m(x_1)$. 
In order to complete the proof, we will need the following lemma.
\begin{lem}\label{lem:jump with Q}
 The path $Q$, defined by the interval $\left[z,m(x_{|I|})\right]$ is such that its endpoints are $3\varepsilon^3k$-dense in $\left[0,\frac{k}{\varepsilon^4}\right]$ in the graph $G[Q]$. 
\end{lem}
\begin{proof}
Denote by $R_1$ the path consisting only of vertex $z$. We now recursively define for each $i< |I|$ a path $R_i$ whose one endpoint is $z$ and the other endpoint $z_i$ lies in either $Q_i^1$ or $Q_i^3$ (see Figure~\ref{fig:jumpswithQ} for an illustration). First, let $z_1=z$.
Suppose $z_i$ is in $Q_i^a$ for some $a\in\{1,3\}$, and let $b\in\{1,3\}\setminus\{a\}$;  let $R_{i+1}$ be the path obtained from $R_{i}$ by concatenating to it the path contained in $Q_i$ which starts at $z_i$, goes through $Q_i^2$ and touches the edge $e^b_{i,i+1}$, and also add the edge $e^b_{i,i+1}$ itself. 

Now we also define paths $R_i'$ for each $i \leq |I|-1$, obtained from $R_i$ as follows. If $z_i\in Q^{a}_i$, again we let $b\neq a$ and $b\in\{1,3\}$. Let $R_i'$ be the path obtained by concatenating with $R_i$ the path starting at $z_i$, going through $Q_i$ until the edge $e^b_{i,|I|}$, then also adding this edge $e^b_{i,|I|}$ itself, together with the path in $Q_{|I|}$ which connects the endpoint of this edge with $m(x_{|I|})$.

Note that the length of two paths $R_i'$ and $R_{i+1}'$ differs by at most $|Q_i|+|Q_{i+1}|+|Q_I|\leq 3\varepsilon^3k$, since the only vertices which belong to exactly one of these two paths are contained in $G[Q_i\cup Q_{i+1}\cup Q_I]$. Furthermore, the length of the first path $R'_1$ obtained by our procedure is at most $|Q_1| + |Q_{|I|}| \leq 2 \eps^3 k$ and the length of the last path $R_{|I|-1}'$ is at least $(|I|-1)\varepsilon^3k/3\geq \frac{k}{\varepsilon^4}$, since it contains all paths $Q_i^2$ for all $i\leq |I|-1$. This implies that the path $Q$, defined as the path between $z=z_1$ and $m(x_{|I|})$ is such that its endpoints are $3\varepsilon^3k$-dense in $\left[0,\frac{k}{\varepsilon^4}\right]$ in the graph $G[Q]$. 
 \end{proof}
 
 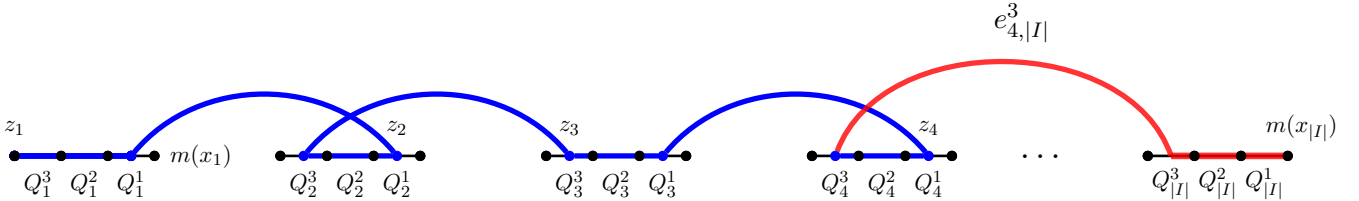
\begin{figure}[ht]
    \centering
    \begin{tikzpicture}[scale=0.62,main node/.style={circle,draw,color=black,fill=black,inner sep=0pt,minimum width=7pt}]
    
    	\tikzset{rectangle/.append style={draw=brown, ultra thick, fill=blue!30}}

    \foreach \i in {1,...,4}
     {
     \node[main node, scale=0.5, color=black, opacity=1] (a\i) at (5.7*\i+1,0){};
     \node[main node, scale=0.5, color=black, opacity=1] (b\i) at (5.7*\i+2,0){};
     
     \node[main node, scale=0.5, color=black, opacity=1] (c\i) at (5.7*\i+3,0){};
     \node[main node, scale=0.5, color=black, opacity=1] (d\i) at (5.7*\i+4,0){};
     
     \draw[line width= 1 pt] (a\i) to (d\i);
     \node[main node, scale=0.5, color=blue, opacity=1] (z1) at (6.7,0){};
  \node[main node, scale=0.5, color=blue, opacity=1] (t1) at (9.2,0){};

    \draw[line width= 2 pt, color=blue] (z1) to (t1); 
     \ifthenelse {\i>1 \and \i<5}
     {
     \pgfmathtruncatemacro\aa{(-1)^\i}
      \node[main node, scale=0.5, color=blue, opacity=1] (z\i) at (5.7*\i+2.5+\aa,0){};
      \node[main node, scale=0.5, color=blue, opacity=1] (t\i) at (5.7*\i+2.5-\aa,0){};
      \draw[line width= 2 pt, color=blue] (z\i) to (t\i);
     }
     {}
      \ifthenelse {\i<5 \and \i>1}
      {
         \pgfmathtruncatemacro\aa{\i-1}
         \draw[line width= 2 pt, color=blue] (z\i) to [bend right = 50] (t\aa);
      }
      {}
    
     }
     
     \foreach \i in{1,...,4}
     {
     \node[main node, scale=0.5, color=black, opacity=1] (a\i) at (5.7*\i+1,0){};
     \node[main node, scale=0.5, color=black, opacity=1] (b\i) at (5.7*\i+2,0){};
     \node[main node, scale=0.5, color=black, opacity=1] (c\i) at (5.7*\i+3,0){};
     \node[main node, scale=0.5, color=black, opacity=1] (d\i) at (5.7*\i+4,0){};
     
     \node[scale=0.8, color=black, opacity=1] (t) at (5.7*\i+1.5,-0.6){$Q_\i^3$};
     \node[scale=0.8, color=black, opacity=1] (t) at (5.7*\i+2.5,-0.6){$Q_\i^2$};
     \node[scale=0.8, color=black, opacity=1] (t) at (5.7*\i+3.5,-0.6){$Q_\i^1$};
     }

     \node[main node, scale=0.5, color=black, opacity=1] (a) at (31,0){};
     \node[main node, scale=0.5, color=black, opacity=1] (b) at (32,0){};
     \node[main node, scale=0.5, color=black, opacity=1] (c) at (33,0){};
     \node[main node, scale=0.5, color=black, opacity=1] (d) at (34,0){};
     
     \node[scale=0.8, color=black, opacity=1] (t) at (31.5,-0.6){$Q_{|I|}^3$};
     \node[scale=0.8, color=black, opacity=1] (t) at (32.5,-0.6){$Q_{|I|}^2$};
     \node[scale=0.8, color=black, opacity=1] (t) at (33.5,-0.6){$Q_{|I|}^1$};

 \draw[line width= 1 pt] (a) to (d);
 \draw[color=red,line width= 2 pt, opacity =0.8] (t4) to [ bend left = 70](31.5,0);
 \draw[color=red,line width= 2.5 pt, opacity=0.8] (34,0) to (31.5,0);
 
 \node[main node, scale=0.5, color=black, opacity=1] (a) at (31,0){};
     \node[main node, scale=0.5, color=black, opacity=1] (b) at (32,0){};
     \node[main node, scale=0.5, color=black, opacity=1] (c) at (33,0){};
     \node[main node, scale=0.5, color=black, opacity=1] (d) at (34,0){};

    \node[scale=0.8, color=black, opacity=1] (t) at (6.7,0.6){$z_1$};
    \node[scale=0.8, color=black, opacity=1] (t) at (14.9,0.6){$z_2$};
    \node[scale=0.8, color=black, opacity=1] (t) at (18.6,0.6){$z_3$};
    \node[scale=0.8, color=black, opacity=1] (t) at (26.3,0.6){$z_4$};
    
    \node[scale=0.8, color=black, opacity=1] (t) at (34.3,0.6){$m(x_{|I|})$};
    
    \node[scale=1.2, color=black, opacity=1] (t) at (28.7,0){$\ldots$};
    
    \node[scale=1, color=black, opacity=1] (t) at (28.3,2.85){$e_{4,|I|}^3$};
     
      \node[scale=0.8, color=black, opacity=1] (t) at (10.7,0){$m(x_1)$};

    \end{tikzpicture} 
    \caption{
    The thick blue path represents path $R_4$, and adding to it the red path creates $R_4'$.
    }
    
    \label{fig:jumpswithQ}
\end{figure}

Let $Q^*$ be the path in $C$ spanned by the interval $[1,z]\cup \left[m(y_{|I|}),N\right]$. By applying Lemma~\ref{lem:easyjump}, we get a path $Q'$ of length at most $2k$ in $G[Q^*]$ with the same endpoints as $Q^*$. Recalling that the pair of vertices $x_{|I|},y_{|I|}$ is connected by paths of all lengths in $\left[\frac{\varepsilon k}{100},\frac{\varepsilon k}{50}\right]$ in the subgraph $G[S]$, and Lemma~\ref{lem:jump with Q} above, we are done by \Cref{obs:combining}. Indeed, we apply it to $G$ with $v_1=z$, $v_2=m(y_{|I|})$ and $v_3=m(x_{|I|})$, while $S_1$ is the set of internal vertices of $Q^*$, $S_2=S$ and $S_3$ are the internal vertices of $Q$ to get all cycle lengths in the interval $\left[(2+\eps) k,\frac{1000}{\varepsilon^2}k\right]$.
\hfill \qedsymbol

\section{Concluding remarks}\label{sec:concludingrem}

In this paper we proved that every Hamiltonian graph on $n \geq 2k^2+o(k^2)$ vertices with independence number $k$ is pancyclic, which is tight up to the $o(k^2)$ error term. Furthermore, our methods allow us to give a short proof of Erd\H{o}s's conjecture that $n = \Omega(k^2)$ vertices are enough for $G$ to be pancyclic. For this, we first note that while getting a bound of $n = \Omega(k^3)$,  Keevash and Sudakov \cite{keevash2010pancyclicity} implicitly proved the following result.
\begin{lem}[\cite{keevash2010pancyclicity}]\label{lem:ks}
There exists a large constant $C$ such that every Hamiltonian graph on $n \geq Ck^2$ vertices with independence number $k$, contains all cycle lengths in $[3,n/C]$.
\end{lem}
\noindent This reduces Erd\H{o}s's conjecture to the following problem. 
\begin{itemize}
    \item[($\ast$)] Does there exist $C' >0$ such that every Hamiltonian graph on $n \geq C'k^2$ vertices with independence number $k$, contains a cycle of length $n-1$? 
\end{itemize}
\noindent Indeed, suppose that the above is true for some large constant $C'$ and let $G$ be a Hamiltonian graph on $n$ vertices with independence number $k$. Then, combining this with the above lemma of Keevash and Sudakov, it follows that if $n \geq CC'k^2$ then $G$ is pancyclic, thus proving Erd\H{o}s's conjecture. Indeed, note that by the lemma above, $G$ contains all cycle lengths up to $n/C$ and one can see that it contains all cycle lengths from $n/C$ to $n$ by iteratively applying the assumption that whenever $n' \geq C'k^2$, there is a cycle of length $n'-1$. 
The previous results 
\cite{keevash2010pancyclicity}, \cite{lee2012hamiltonicity} and \cite{dankovics2020low} are all improvements towards question ($\ast$) above. As discussed in the beginning of Section \ref{sec:upperrange}, applying Proposition \ref{lem:jumpwithzigzag} with $c=1$ and $U = \emptyset$ solves this problem in the following stronger form.
\begin{thm}\label{thm:outlinen-1}
Every Hamiltonian graph on $n > 2k^2+2k$ vertices with independence number $k$, contains a cycle of length $n-1$.
\end{thm}
Let us note that although \Cref{thm:outlinen-1} shows the existence of a cycle of length $n-1$ already with $n>2k^2+2k$, this is not sufficient to prove that $n>2k^2+\varepsilon k^2$ implies pancyclicity. At this threshold, \Cref{lem:ks} does not apply, so one needs a different argument to find the cycle lengths in the interval $[3,2k^2+2k]$. It turns out that in this setting, the cycle lengths which are hardest to find are those around $2k$, that is, precisely the cycle lengths which are missed by the lower bound construction given in the introduction. Finding them is the most technical part of our proof, given in \Cref{sec:middlerange}.

A very interesting open question is to understand the best bound on the number of vertices $n$ in ($\ast$) which guarantees the cycle of length $n-1$. Here the answer might be linear in $n$ as the following question asked in \cite{keevash2010pancyclicity} suggests.
\begin{prob}
Does there exist a constant $C$ such that every Hamiltonian graph with independence number $k$ and $n \geq Ck$ vertices contains a cycle of length $n-1$?
\end{prob}

\noindent\textbf{Acknowledgements.} We thank the referee for their careful reading of the paper and for useful comments.

\end{document}